\documentclass[reqno, a4paper]{amsart}

\usepackage[T1]{fontenc}
\usepackage[british]{babel}
\usepackage{graphicx}
\usepackage{amsfonts}
\usepackage{amssymb}
\usepackage{amsthm}
\usepackage{amsmath}
\usepackage[all]{xy}
\usepackage{hyperref}
\usepackage{mathbbol}
\usepackage{mathabx}

\theoremstyle{plain}
\newtheorem{theorem}[subsection]{Theorem}
\newtheorem{lemma}[subsection]{Lemma}
\newtheorem{proposition}[subsection]{Proposition}

\theoremstyle{definition}
\newtheorem{definition}[subsection]{Definition}
\newtheorem{remark}[subsection]{Remark}
\newtheorem{example}[subsection]{Example}

\newcommand{\defn}{\textbf}

\newcommand{\UCE}{condition~{\rm (UCE)}}
\newcommand{\UCEspecial}{condition~(UCE)}
\newcommand{\UCEcapital}{Condition~(UCE)}
\renewcommand{\b}{\ensuremath{I}}
\newcommand{\cc}{\ensuremath{J}}

\newcommand{\meet}{\ensuremath{\wedge}}
\newcommand{\comp}{\raisebox{0.2mm}{\ensuremath{\scriptstyle{\circ}}}}
\newcommand{\pr}{\ensuremath{\pi}}
\newcommand{\tensor}{\ensuremath{\otimes}}
\newcommand{\To}{\Rightarrow}

\newcommand{\Ker}{\ensuremath{\mathrm{Ker}}}
\renewcommand{\H}{\ensuremath{\mathrm{H}}}
\renewcommand{\hom}{\ensuremath{\mathrm{Hom}}}

\newcommand{\U}{\ensuremath{\mathrm{U}}}

\newcommand{\A}{\ensuremath{\mathcal{A}}}
\newcommand{\B}{\ensuremath{\mathcal{B}}}
\newcommand{\C}{\ensuremath{\mathcal{C}}}

\newcommand{\Ab}{\ensuremath{\mathsf{Ab}}}
\newcommand{\ab}{\ensuremath{\mathsf{ab}}}

\newcommand{\Centr}{\ensuremath{\mathsf{Centr}}}
\newcommand{\CExt}{\ensuremath{\mathsf{CExt}}}
\newcommand{\Ext}{\ensuremath{\mathsf{Ext}}}

\newcommand{\Leibniz}{\ensuremath{\mathsf{Leib}}}
\newcommand{\Lie}{\ensuremath{\mathsf{Lie}}}
\newcommand{\Mod}{\ensuremath{\mathsf{Mod}}}
\newcommand{\NAAlg}{\ensuremath{\mathsf{NAAlg}}}
\newcommand{\PXMod}{\ensuremath{\mathsf{PXMod}}}

\newcommand{\Vect}{\ensuremath{\mathsf{Vect}}}
\newcommand{\XMod}{\ensuremath{\mathsf{XMod}}}

\newcommand{\K}{\ensuremath{\mathbb{K}}}
\newcommand{\Z}{\ensuremath{\mathbb{Z}}}

\newdir{>>}{{}*!/3.5pt/:(1,-.2)@^{>}*!/3.5pt/:(1,+.2)@_{>}*!/7pt/:(1,-.2)@^{>}*!/7pt/:(1,+.2)@_{>}}
\newdir{ >>}{{}*!/8pt/@{|}*!/3.5pt/:(1,-.2)@^{>}*!/3.5pt/:(1,+.2)@_{>}}
\newdir{ |>}{{}*!/-3.5pt/@{|}*!/-8pt/:(1,-.2)@^{>}*!/-8pt/:(1,+.2)@_{>}}
\newdir{ >}{{}*!/-8pt/@{>}}
\newdir{>}{{}*:(1,-.2)@^{>}*:(1,+.2)@_{>}}
\newdir{<}{{}*:(1,+.2)@^{<}*:(1,-.2)@_{<}}

\def\pullback{
 \ar@{-}[]+R+<6pt,-1pt>;[]+RD+<6pt,-6pt>%
 \ar@{-}[]+D+<1pt,-6pt>;[]+RD+<6pt,-6pt>}
 
\def\splitpullback{%
 \ar@{-}[]+R+<6pt,-.51ex>;[]+RD+<6pt,-6pt>%
 \ar@{-}[]+D+<.51ex,-6pt>;[]+RD+<6pt,-6pt>}

\begin{document}

\title[Universal central extensions in semi-abelian categories]{Universal central extensions\\ in semi-abelian categories}

\author{Jos\'e Manuel Casas}
\author{Tim Van~der Linden}

\email{jmcasas@uvigo.es}
\email{tim.vanderlinden@uclouvain.be}

\address{Dpto.\ de Matem\'atica Aplicada I, Universidad de Vigo, Escola Enxe\~nar\'ia Forestal, Campus Universitario A Xunquiera, E-36005 Pontevedra, Spain}
\address{Institut de recherche en math\'ematique et physique, Universit\'e catholique de Louvain, chemin du cyclotron~2 bte~L7.01.02, B-1348 Louvain-la-Neuve, Belgium}
\address{CMUC, University of Coimbra, 3001--454 Coimbra, Portugal}

\thanks{The first author's research was supported by Ministerio de Ciencia e Innovaci\'on (grant number MTM2009-14464-C02-02, includes European FEDER support) and by Xunta de Galicia (grant number Incite09 207 215 PR). The second author works as \emph{charg\'e de recherches} for Fonds de la Recherche Scientifique--FNRS. His research was supported by Centro de Matem\'a\-tica da Universidade de Coimbra and by Funda\c c\~ao para a Ci\^encia e a Tecnologia (grant number SFRH/BPD/38797/2007). We thank the University of Coimbra, the University of Vigo and the Banff International Research Station for their kind hospitality.}

\begin{abstract}
Basing ourselves on Janelidze and Kelly's general notion of central extension, we study universal central extensions in the context of semi-abelian categories. We consider a new fundamental condition on composition of central extensions and give examples of categories which do, or do not, satisfy this condition.
\end{abstract}



\maketitle

\section*{Introduction}

In this article we explain how, using Janelidze and Kelly's general notion of central extension~\cite{Janelidze-Kelly}, the classical theory of universal central extensions may be considered in the context of semi-abelian categories~\cite{Janelidze-Marki-Tholen}. Our main point is that, while \emph{most} of the results valid for groups and Lie algebras (see, for instance, \cite{Milnor, Weibel}) generalise without any difficulty to extensions, central with respect to a chosen Birkhoff subcategory $\B$ of a semi-abelian category $\A$, this setting turns out to be too weak for some of the most basic results, valid in the classical examples, to hold---even when $\B=\Ab(\A)$ is determined by the abelian objects in $\A$.

We have to impose an additional condition which we chose to call the \emph{universal central extension} \emph{\UCE}, asking that for $\Ab(\A)$-central extensions $f\colon{B\to A}$ and $g\colon{C\to B}$ also the composite extension $f\comp g$ is $\Ab(\A)$-central, as soon as $B$ is an $\Ab(\A)$-perfect object. (Recall that in general, central extensions need not compose.) Under \UCE\ and, as it turns out, only then, standard recognition results such as Theorem~\ref{Theorem-Homology} hold. Furthermore, as Example~\ref{Counter1} shows, \UCE\ is not automatic: there exist semi-abelian categories which do not satisfy it. This immediately gives rise to the following question, which is not yet fully answered in the present paper: \emph{When does \UCE\ hold?} We can give examples and counterexamples, but thus far there is no elementary characterisation. This problem---of finding good minimal hypotheses for \UCE---is the subject of current work-in-progress~\cite{GrayVdL1}.

\section{Basic definitions and first results}\label{Section-Basic-Definitions}

\subsection{Semi-abelian categories}
A category is \defn{semi-abelian} when it is pointed, Barr exact and Bourn protomodular with binary coproducts~\cite{Janelidze-Marki-Tholen}. We recall that a pointed and regular category is \defn{Bourn protomodular}~\cite{Bourn1991} if and only if the \defn{(Regular) Short Five Lemma} holds: this means that for any commutative diagram such as~\eqref{Short-Five-Lemma} below where $f$ and $f'$ are regular epimorphisms, $k$ and $a$ being isomorphisms implies that $b$ is an isomorphism. It is well known that all varieties of $\Omega$-groups~\cite{Higgins} are semi-abelian categories.

\begin{lemma}\cite{Bourn1991, Bourn2001}\label{Lemma-Pullback}
Consider a morphism of short exact sequences
\begin{equation}\label{Short-Five-Lemma}
\vcenter{\xymatrix{\Ker(f') \ar@{{ |>}->}[r]^-{\ker(f')} \ar[d]_-k & B' \ar@{-{ >>}}[r]^-{f'} \ar[d]_-b & A' \ar[d]^-a \\ \Ker(f) \ar@{{ |>}->}[r]_-{\ker(f)} & B \ar@{-{ >>}}[r]_-{f} & A.}}
\end{equation}
\begin{enumerate}
\item The right hand side square $f\comp b=a\comp f'$ is a pullback iff $k$ is an isomorphism.
\item The left hand side square $\ker(f)\comp k=b\comp \ker(f')$ is a pullback iff $a$ is mono.
\end{enumerate}
\end{lemma}

The first statement implies that any pullback square between regular epimorphisms (that is, any square $f\comp b=a\comp f'$ as in~\eqref{Short-Five-Lemma}) is a pushout. It is also well known that the regular image of a kernel is a kernel~\cite{Janelidze-Marki-Tholen}. In any semi-abelian category, classical homological lemma's such as the Snake Lemma and the $3\times3$ Lemma are valid; for further details and many other results we refer the reader to~\cite{Janelidze-Marki-Tholen, Borceux-Bourn}.

\subsection{Birkhoff subcategories}
The notion of central extension introduced in~\cite{Janelidze-Kelly} is defined with respect to a chosen subcategory $\B$ of the base category $\A$: a \defn{Birkhoff subcategory} $\B$ of a semi-abelian category~$\A$ is a full and reflective subcategory, closed under subobjects and regular quotients. We denote the left adjoint by $\b\colon{\A\to \B}$ and write the components of its unit $\eta_{A}\colon{A\to \b(A)}$. A Birkhoff subcategory of a variety of universal algebras is the same thing as a subvariety. Throughout the text, we fix a Birkhoff subcategory $\B$ of a semi-abelian category~$\A$.

\subsection{Extensions and central extensions}\label{Subsection-Central-Extensions}

An \defn{extension} in $\A$ is a regular epimorphism. A morphism of extensions is a commutative square between them, and thus we obtain the category $\Ext(\A)$ of extensions in $\A$. 

With respect to $\B\subseteq \A$, there are notions of \emph{trivial}, \emph{normal} and \emph{central} extension. An extension $f\colon{B\to A}$ in $\A$ is said to be \defn{trivial} if and only if the induced square
\begin{equation}\label{Birkhoff-Square}
\vcenter{\xymatrix{B \ar[r]^-{f} \ar[d]_-{\eta_{B}} & A \ar[d]^-{\eta_{A}}\\
\b(B) \ar[r]_-{\b(f)} & \b(A)}}
\end{equation}
is a pullback. The extension $f$ is \defn{normal} if and only if one of the projections $f_{0}$,~$f_{1}$ in the kernel pair $(B\times_A B,f_{0},f_{1})$ of $f$ is trivial. Finally, $f$ is said to be \defn{central} if and only if there exists an extension $g\colon{C\to A}$ such that the pullback~$g^{*}(f)$ of $f$ along~$g$ is~trivial.

Clearly, every normal extension is central; in the present context, the converse also holds, and thus the concepts of normal and central extension coincide. Furthermore, a split epimorphism is a trivial extension if and only if it is central~\cite[Theorem~4.8]{Janelidze-Kelly}. Finally, central extensions are pullback-stable~\cite[Proposition~4.3]{Janelidze-Kelly}.

Since we shall often be considering several Birkhoff subcategories of a given category at the same time, we usually indicate which one we mean by prefixing, as in ``$\B$-central'', ``$I$-trivial'', etc.

\subsection{Perfect objects}
An object $P$ of $\A$ is called \defn{perfect} when $\b(P)$ is the zero object $0$ of $\B$. If $f\colon{B\to A}$ is an extension and~$B$ is $\B$-perfect then so is $A$, because the reflector $\b$ preserves regular epimorphisms, and a regular quotient of zero is zero.

\begin{lemma}\label{Lemma-Perfect-and-Central}
Let $P$ be a $\B$-perfect object and $f\colon{B\to A}$ an extension.
\begin{enumerate}
\item If $f$ is $\B$-trivial then the map
\[
\hom(P,f)=f\comp(-)\colon{\hom(P,B)\to \hom(P,A)}
\]
is a bijection;
\item if $f$ is $\B$-central then $\hom(P,f)$ is an injection.
\end{enumerate}
If $\hom(P,f)$ is an injection for every $\B$-trivial extension $f$ then $P$ is $\B$-perfect.
\end{lemma}
\begin{proof}
The extension $f$ being $\B$-trivial means that the square~\eqref{Birkhoff-Square} is a pullback. If $b_{0}$, $b_{1}\colon{P\to B}$ are morphisms such that $f\comp b_{0}=f\comp b_{1}$, then $b_{0}$ is equal to $b_{1}$ by the uniqueness in the universal property of this pullback: indeed also $\eta_{B}\comp b_{0}=\b(b_{0})\comp\eta_{P}=0=\b(b_{1})\comp\eta_{P}=\eta_{B}\comp b_{1}$. Thus we see that $\hom(P,f)$ is injective.
This map is also surjective, since any morphism $a\colon{P\to A}$ is such that $\eta_{A}\comp a=\b(a)\comp \eta_{P}=0=\b(f)\comp 0$ and thus induces a morphism $b\colon{P\to B}$ for which $f\comp b=a$.

Statement 2 follows from 1 because the functor $\hom(P,-)$ preserves kernel pairs, and a map is an injection if and only if its kernel pair projections are bijections.

As to the converse: the morphism $!_{\b(P)}\colon{\b(P)\to 0}$ is a $\B$-trivial extension; since $!_{\b(P)}\comp \eta_{P}=0=!_{\b(P)}\comp 0$, the assumption implies that $\eta_{P}$ is zero, which means that $P$ is $\B$-perfect.
\end{proof}

\subsection{Universal central extensions}
For any object $A$ of $\A$, let $\Centr_{\B}(A)$ or $\Centr_{I}(A)$ denote the category of all $\B$-central extensions of $A$: the full subcategory of the slice category $(\A\downarrow A)$ determined by the central extensions. A (weakly) initial object of this category $\Centr_{\B}(A)$ is called a \defn{(weakly) universal central extension} of~$A$. A $\B$-central extension $u\colon{U\to A}$ is weakly universal when for every $\B$-central extension $f\colon{B\to A}$ there exists a morphism $\overline{f}$ from $u$ to $f$, that is, such that $f\comp \overline{f}=u$. Furthermore, $u$ is universal when this induced morphism $\overline{f}$ is unique. Note also that, up to isomorphism, an object admits at most one universal $\B$-central extension.

\begin{lemma}\label{UCE-then-Perfect}
If $u\colon{U\to A}$ is a universal $\B$-central extension then the objects $U$ and~$A$ are $\B$-perfect.
\end{lemma}
\begin{proof}
Since the first projection $\pr_{A}\colon {A\times \b(U)\to A}$ is a trivial extension, by~\ref{Subsection-Central-Extensions} it is central. By the hypothesis that $u$ is universal, there exists just one morphism $\langle u,v\rangle\colon {U\to A\times \b(U)}$ such that $\pr_{A}\comp\langle u,v\rangle=u$. But then $0\colon {U\to \b(U)}$ is equal to $\eta_{U}\colon {U\to \b(U)}$, and $\b(U)=0$. Since a regular quotient of a perfect object is perfect, this implies that both $U$ and $A$ are $\B$-perfect.
\end{proof}

\begin{proposition}\label{Proposition-Universal-vs-Weak-Universal}
Let $\A$ be a semi-abelian category and~$\B$ a Birkhoff subcategory of $\A$. Let $u\colon{U\to A}$ be a $\B$-central extension. Between the following conditions, the implications $1\Leftrightarrow 2 \Leftrightarrow 3\Rightarrow 4\Leftrightarrow 5$ hold:
\begin{enumerate}
\item $U$ is $\B$-perfect and every $\B$-central extension of $U$ splits;
\item $U$ is $\B$-perfect and projective with respect to all $\B$-central extensions;
\item for every $\B$-central extension $f\colon{B\to A}$, the map
\[
\hom(U,f)\colon{\hom(U,B)\to \hom(U,A)}
\]
is a bijection; 
\item $U$ is $\B$-perfect and $u$ is a weakly universal $\B$-central extension;
\item $u$ is a universal $\B$-central extension.
\end{enumerate}
\end{proposition}
\begin{proof}
Suppose that 1 holds. To prove 2, let $f\colon{B\to A}$ be a $\B$-central extension and $g\colon{U\to A}$ a morphism. Then the pullback $g^{*}(f)\colon{\overline{B}\to U}$ of $f$ along $g$ is still $\B$-central; hence $g^{*}(f)$ admits a splitting $s\colon{U\to \overline{B}}$, and $(f^{*}(g))\comp s$ is the required morphism ${g\to f}$. Conversely, given a $\B$-central extension $f\colon{B\to U}$, the projectivity of~$U$ yields a morphism $s\colon{U\to B}$ such that $f\comp s=1_{U}$. 

Conditions 2 and 3 are equivalent by Lemma~\ref{Lemma-Perfect-and-Central}.

Condition 3 implies condition 5: given a $\B$-central extension $f\colon{B\to A}$ of $A$, there exists a unique morphism $\overline{f}\colon{U\to B}$ that satisfies $f\comp \overline{f}=u$.

Finally, 4 and 5 are equivalent by Lemma~\ref{Lemma-Perfect-and-Central} and Lemma~\ref{UCE-then-Perfect}.
\end{proof}

\begin{remark}
To prove that condition 4 implies 3 we would require $U$ itself to admit a universal $\B$-central extension, which need not be the case in the present context. In fact, even if such a universal $\B$-central extension of $U$ does exist, then the above five conditions may or may not be equivalent: see Section~\ref{Section-UCE}.
\end{remark}

\section{Constructing universal central extensions}\label{Section-UCE-Construction}
Our aim is now to prove that every perfect object admits a universal central extension. To do so, we need to assume that the surrounding category has enough projectives: they will give us weakly universal central extensions.

\subsection{Commutators and centralisation}\label{Subsection-Commutators-and-Centralisation}
The kernel $\mu$ of the unit $\eta$ of $\b\colon{\A\to \B}$ gives rise to a ``zero-dimensional'' commutator: for any object~$A$ of~$\A$, the bottom row in~\eqref{Birkhoff} is a short exact sequence in $\A$; hence $A$ is an object of $\B$ if and only if $[A,A]_{\B}=0$. On the other hand, an object~$A$ of~$\A$ is $\B$-perfect precisely when $[A,A]_{\B}=A$. This construction defines a functor $[-,-]_{\B}\colon{\A\to \A}$ and a natural transformation~$\mu\colon{[-,-]_{\B}\To 1_{\A}}$. The functor $[-,-]_{\B}$ preserves regular epimorphisms; we recall the argument. Given a regular epimorphism $f\colon{B\to A}$, by the Birkhoff property, the induced square of regular epimorphisms on the right
\begin{equation}\label{Birkhoff}
\vcenter{\xymatrix{0 \ar[r] & [B,B]_{\B} \ar@{{ |>}->}[r]^-{\mu_{B}} \ar@{.>}[d]_-{[f,f]_{\B}} & B \ar@{ >>}[d]_-{f} \ar@{ >>}[r]^-{\eta_{B}} & \b(B) \ar@{ >>}[d]^-{\b(f)} \ar[r] & 0\\
0 \ar[r] & [A,A]_{\B} \ar@{{ |>}->}[r]_-{\mu_{A}} & A \ar@{ >>}[r]_-{\eta_{A}} & \b(A) \ar[r] & 0}}
\end{equation}
is a pushout---which is equivalent to $[f,f]_{\B}$ being regular epic~\cite[Corollary 5.7]{EverVdL1}.

Lemma~\ref{Lemma-Pullback} implies that an extension $f\colon{B\to A}$ is $\B$-central if and only if either one of the morphisms $[f_{0},f_{0}]_{\B}$, $[f_{1},f_{1}]_{\B}$ is an isomorphism, which, because they have a common splitting, happens exactly when they coincide, $[f_{0},f_{0}]_{\B}=[f_{1},f_{1}]_{\B}$. Hence the kernel $[K,B]_{\B}$ of $[f_{0},f_{0}]_{\B}$ measures how far $f$ is from being central: indeed, $f$ is $\B$-central if and only if $[K,B]_{\B}$ is zero.
\[
\xymatrix{& [K,B]_{\B} \ar@{{ |>}->}[d]_-{\ker [f_{0},f_{0}]_{\B}} \ar@{{ |>}.>}[ddr] \\
0 \ar[r] & [B\times_A B,B\times_A B]_{\B} \ar@{{ |>}->}[r]^-{\mu_{B\times_A B}} \ar@<-.5ex>@{ >>}[d]_-{[f_{0},f_{0}]_{\B}} \ar@<.5ex>@{ >>}[d]^-{[f_{1},f_{1}]_{\B}} & B\times_A B \ar@<-.5ex>@{ >>}[d]_-{f_{0}} \ar@<.5ex>@{ >>}[d]^-{f_{1}} \ar@{ >>}[r]^-{\eta_{B\times_A B}} & \b(B\times_A B) \ar@<-.5ex>@{ >>}[d]_-{\b(f)_{0}} \ar@<.5ex>@{ >>}[d]^-{\b(f)_{1}} \ar[r] & 0\\
0 \ar[r] & [B,B]_{\B} \ar@{{ |>}->}[r]_-{\mu_{B}} & B \ar@{ >>}[r]_-{\eta_{B}} & \b(B) \ar[r] & 0}
\]

\begin{remark}\label{Remark-Subobject}
This explains, for instance, why a sub-extension of a central extension is central. It is worth recalling here that a morphism of extensions $(b,a)$ 
\begin{equation*}\label{Morphisms}
\vcenter{\xymatrix{B' \ar@{-{ >>}}[d]_-{f'} \ar[r]^-{b} & B \ar@{-{ >>}}[d]^-{f}\\
A' \ar[r]_-{a} & A}}
\end{equation*}
is a monomorphism if and only if $b$ is.
\end{remark}

The ``one-dimensional'' commutator $[K,B]_{\B}$ may be considered as a normal subobject of $B$ via the composite $\mu_{B}\comp [f_{1},f_{1}]_{\B}\comp \ker [f_{0},f_{0}]_{\B}\colon{[K,B]_{\B}\to B}$ (see the diagram above). Thus we obtain the left adjoint $\b_1\colon \Ext(\A)\to \CExt_{\B}(\A)$, where $\CExt_{\B}(\A)$ is the full reflective subcategory of $\Ext(\A)$ determined by the $\B$-central extensions. Given an extension $f\colon{B\to A}$ with kernel $K$, its \defn{centralisation} $\b_{1}(f)\colon{B/[K,B]_{\B}\to A}$ is obtained through the diagram with exact rows
\[
\xymatrix{0 \ar[r] & [K,B]_{\B} \ar@{{ |>}->}[r] \ar@{ >>}[d] & B \ar@{ >>}[r] \ar@{ >>}[d]_-{f} & \tfrac{B}{[K,B]_{\B}} \ar[r] \ar@{ >>}[d]^-{\b_{1}(f)} & 0\\
& 0 \ar@{{ |>}->}[r] & A \ar@{=}[r] & A \ar[r] & 0.}
\]

\subsection{Existence of a weakly universal central extension}\label{Subsection-Enough-Projectives}
We say that $\A$ \defn{has weakly universal central extensions} (for some Birkhoff subcategory $\B$ of $\A$) when every object of $\A$ admits a weakly universal $\B$-central extension. This happens, for instance, when $\A$ has enough (regular) projectives, so that for any object~$A$ of~$\A$, there exists a regular epimorphism $f\colon{B\to A}$ with $B$ projective, a \defn{(projective) presentation} of $A$.

\begin{lemma}\label{Lemma-Weakly-UCE}
If the category $\A$ is semi-abelian with enough projectives then it has weakly universal central extensions for any Birkhoff subcategory $\B$.
\end{lemma}
\begin{proof}
Given an object $A$ of $\A$, the category $\Centr_{\B}(A)$ has a weakly initial object: given a projective presentation $f\colon{B\to A}$ with kernel $K$, its centralisation $\b_{1}(f)\colon{B/[K,B]_{\B}\to A}$ is weakly initial. Indeed, any $\B$-central extension $g\colon {C\to A}$ induces a morphism ${\b_{1}(f)\to g}$ in $\Centr_{\B}(A)$, because the object $B$ is projective.
\end{proof}

\subsection{Baer invariants}\label{Subsection-Baer-Invariants}\label{Existence of universal central extension}
Let $A$ be an object of $\A$ and $f\colon{B\to A}$ a projective presentation with kernel~$K$. The induced objects
\[
\frac{[B,B]_{\B}}{[K,B]_{\B}}\qquad\qquad \text{and} \qquad\qquad \frac{K\meet [B,B]_{\B}}{[K,B]_{\B}}
\]
are independent of the chosen projective presentation of $A$ as explained for instance in~\cite{EverVdL1}. The object ${(K\meet [B,B]_{\B})}/{[K,B]_{\B}}$ is called (the Hopf formula for) the \defn{second homology object} of $A$ \defn{(with coefficients in $\B$)} and is written $\H_{2}(A,\b)$. We write $\U(A,\b)$ for the object ${[B,B]_{\B}}/{[K,B]_{\B}}$ and $\H_{1}(A,\b)$ for $\b(A)$.

The objects $\H_{2}(A,\b)$ and $\H_{1}(A,\b)$ are genuine homology objects: if~$\A$ is a semi-abelian monadic category then they may be computed using comonadic homology as in~\cite{EverVdL2}, and in any case, they fit into the homology theory worked out in~\cite{EverHopf}.

The Baer invariants from~\ref{Subsection-Baer-Invariants} may also be considered for all weakly universal $\B$-central extensions of an object~$A$. Since, for any weakly universal $\B$-central extension $f\colon{B\to A}$ with kernel $K$, the commutator $[K,B]_{\B}$ is zero, this implies that the objects
\[
[B,B]_{\B}\qquad\qquad \text{and} \qquad\qquad K\meet [B,B]_{\B}
\]
are independent of the chosen weakly universal central extension of $A$. (Here, as in~\cite{Janelidze:Hopf}, the Hopf formula becomes $\H_{2}(A,\b)=K\meet [B,B]_{\B}$. Also note that $\U(A,\b)=[B,B]_{\B}$.)

\subsection{The perfect subobject}\label{Perfect-Subobject}
When there are weakly universal central extensions, any central extension of a perfect object contains a subobject with a perfect domain. We prove this in two steps: first for weakly universal central extensions, then in general. This implies that any perfect object admits a universal central extension when weakly universal central extensions exist---a general version of Proposition~4.1 in~\cite{Gran-VdL}.

\begin{lemma}\label{Lemma-Perfect-Subobject}
Suppose $\A$ is a semi-abelian category with a Birkhoff subcategory~$\B$.
Then any weakly universal $\B$-central extension of a $\B$-perfect object contains a subobject with a $\B$-perfect domain.
\end{lemma}
\begin{proof}
Let $f\colon{B\to A}$ be a weakly universal $\B$-central extension of a $\B$-perfect object~$A$. Since $\mu_{A}$ is an isomorphism and $[f,f]_{\B}$ is a regular epimorphism, the morphism $f\comp\mu_{B}=\mu_{A}\comp [f,f]_{\B}$ in the induced diagram with exact rows
\[
\xymatrix{
0 \ar[r] & K\meet [B,B]_{\B} \pullback \ar@{{ |>}->}[r] \ar@{{ |>}->}[d] & [B,B]_{\B} \ar@{{ |>}->}[d]^-{\mu_{B}} \ar@{-{ >>}}[r]^-{f\circ \mu_{B}} & A \ar[r] \ar@{=}[d] & 0\\
0\ar[r] & {K} \ar@{{ |>}->}[r] & {B} \ar@{ >>}[r]_-{f} & A \ar[r] & 0}
\]
is also a regular epimorphism. The extension $f\comp \mu_{B}$ is $\B$-central as a subobject of the $\B$-central extension $f$; its weak universality is clear. By~\ref{Existence of universal central extension} the object $[B,B]_{\B}$ is $\B$-perfect, because the $\B$-central extensions $f\comp \mu_{B}$ and $f$ are both weakly universal, so that $[B,B]_{\B}\cong[[B,B]_{\B},[B,B]_{\B}]_{\B}$.
\end{proof}

\begin{lemma}\label{Lemma-Commutator-Perfect}
Let $\A$ be a semi-abelian category with weakly universal central extensions for a Birkhoff subcategory $\B$ of $\A$. If $f\colon {B\to A}$ is a $\B$-central extension of a~$\B$-perfect object~$A$, then $[B,B]_{\B}$ is also $\B$-perfect.
\end{lemma}
\begin{proof}
The object $B$ admits a weakly universal $\B$-central extension $v\colon {V\to B}$; then the centralisation $w\colon{W\to A}$ of the resulting composite $f\comp v$ is a weakly universal $\B$-central extension. Indeed, given any $\B$-central extension $g\colon {C\to A}$, there is a factorisation $\overline{f^{*}(g)}\colon{V\to B\times_A C}$ of $v$ through the pullback $f^{*}(g)\colon{B\times_A C\to B}$ of $g$ along~$f$, and then the composite $(g^{*}(f))\comp (\overline{f^{*}(g)})\colon {V\to C}$ yields the needed morphism ${w\to g}$ by the universal property of the centralisation functor.

The arrow ${W\to B}$ universally induced by $v$ is regular epic, hence so is its restriction ${[W,W]_{\B}\to [B,B]_{\B}}$; but a regular quotient of a perfect object is perfect.
\end{proof}

\begin{theorem}\label{Theorem-Universal-Central-Extension}
Let $\A$ be a semi-abelian category with enough projectives and $\B$ a Birkhoff subcategory of $\A$. An object $A$ of $\A$ is $\B$-perfect if and only if it admits a universal $\B$-central extension. Moreover, this universal $\B$-central extension may be chosen in such a way that it occurs in a short exact sequence
\[
\xymatrix{0 \ar[r] & \H_{2}(A,\b) \ar@{{ |>}->}[r] & \U(A,\b) \ar@{ >>}[r]^-{u^{\b}_{A}} & A \ar[r] & 0.}
\]
\end{theorem}
\begin{proof}
If an object admits a universal $\B$-central extension then it is $\B$-perfect by Lemma~\ref{UCE-then-Perfect}. Conversely, let $f\colon {B\to A}$ be a weakly universal central extension of a $\B$-perfect object $A$ (Lemma~\ref{Lemma-Weakly-UCE}). Then by Lemma~\ref{Lemma-Perfect-Subobject} it admits a (weakly universal central) subobject with a $\B$-perfect domain. By Proposition~\ref{Proposition-Universal-vs-Weak-Universal}, this subobject is also universal. The shape of the short exact sequence follows from~\ref{Existence of universal central extension}.
\end{proof}

\begin{proposition}\label{Proposition-Quotient-of-UCE}
Let $\A$ be a semi-abelian category with enough projectives and $\B$ a Birkhoff subcategory of $\A$. If $f\colon {B\to A}$ is a $\B$-central extension with a $\B$-perfect domain~$B$, then $f$ is a quotient of a universal $\B$-central extension.
\end{proposition}
\begin{proof}
The construction in the proof of Theorem~\ref{Theorem-Universal-Central-Extension} may be adapted to the given extension $f$ in such a way that the resulting morphism ${u\to f}$ is a regular epimorphism. We take a projective presentation $p\colon{P\to B}$ and use the composite~$f\comp p\colon {P\to A}$ as a projective presentation of $A$. After centralisation we obtain a weakly universal $\B$-central extension $v\colon {V\to A}$ as in Lemma~\ref{Lemma-Weakly-UCE} and a regular epic comparison ${v\to f}$. Using that $B$ is $\B$-perfect, passing to the perfect subobject as in Lemma~\ref{Lemma-Commutator-Perfect} gives us the needed universal $\B$-central extension $u\colon {U\to A}$ together with the induced comparison morphism ${v\to f}$. This morphism is still a regular epimorphism by the Birkhoff property of $\B$ (see Subsection~\ref{Subsection-Commutators-and-Centralisation}).
\end{proof}

\subsection{Universal central extensions and abelianisation}
It is worth remarking here that a universal $\B$-central extension is always central in an absolute sense, namely, with respect to the abelianisation functor $\ab\colon{\A\to \Ab(\A)}$. Here $\Ab(\A)$ is the Birkhoff subcategory of $\A$ consisting of all objects that admit an internal abelian group structure; see, for instance, \cite{Bourn-Gran}.

\begin{proposition}\label{Proposition-Abelian-Result}
Let $\A$ be a semi-abelian category and $\B$ a Birkhoff subcategory of $\A$. If $f\colon {B\to A}$ is a $\B$-central extension with a $\B$-perfect domain~$B$, then $f$ is $\Ab(\A)$-central. In particular, universal $\B$-central extensions are always $\Ab(\A)$-central.
\end{proposition}
\begin{proof}
We have $B\cong [B,B]_{\B}$ since $B$ is $\B$-perfect; $[B,B]_{\B}\cong [{B\times_A B},{B\times_A B}]_{\B}$ because $f$ is $\B$-central. Hence the diagonal ${B\to B\times_A B}$, being isomorphic to $\mu_{B}\colon{[B\times_A B,B\times_A B]_{\B}\to B\times_A B}$, is a kernel. By Proposition~3.1 in~\cite{Bourn-Gran}, this implies that $f$ is $\Ab(\A)$-central. Finally, if $f\colon{B\to A}$ is a universal $\B$-central extension then $B$ is $\B$-perfect.
\end{proof}

\section{Nested Birkhoff subcategories}\label{Section-Nested-Birkhoff-Subcategories}
We now consider the situation where a Birkhoff subcategory $\B$ of a semi-abelian category $\A$ has a further Birkhoff subcategory $\C$ so that they form a chain of nested semi-abelian categories with enough projectives, $\C\subset\B\subset\A$. For instance, $\C$ could be $\Ab(\A)$ as in Theorem~\ref{Theorem-Homology} below. Then there is a commutative triangle of left adjoint functors (all right adjoints are inclusions):
\begin{equation}\label{Triangle}
\vcenter{\xymatrix@!0@=4.5em{{\A} \ar[rr]^-{\b} \ar[dr]_-{\cc\b} && \B \ar[dl]^-{\cc}\\
& \C}}
\end{equation}

\begin{lemma}\label{Lemma-Nesting}
Under the given circumstances,
\begin{enumerate}
\item an object of $\B$ is $\cc$-perfect if and only if it is $\cc\b$-perfect;
\item an extension in $\B$ is $\cc$-central if and only if it is $\cc\b$-central;
\item an extension of $\A$ is $\B$-central as soon as it is $\cc\b$-central.
\end{enumerate}
\end{lemma}
\begin{proof}
If $B$ is an object of $\B$ then $\cc(B)=\cc\b(B)$, which proves the first statement. As for the second statement, an extension $f\colon{B\to A}$ in $\B$ is $\cc$-central if and only if the square in the diagram on the left
\[
\vcenter{\xymatrix{B\times_A B \ar@{-{ >>}}[r]^-{f_{0}} \ar@{-{ >>}}[d]_-{\eta^{\cc}_{B\times_A B}} & B \ar@{-{ >>}}[r]^-{f} \ar@{-{ >>}}[d]^-{\eta^{\cc}_{B}} & A\\
\cc(B\times_A B) \ar@{-{ >>}}[r]_-{\cc(f_{0})} & \cc(B)}}
\qquad\qquad
\vcenter{\xymatrix{B\times_A B \ar@{-{ >>}}[r]^-{f_{0}} \ar@{-{ >>}}[d]_-{\eta^{\cc\b}_{B\times_A B}} & B \ar@{-{ >>}}[d]^-{\eta^{\cc\b}_{B}}\\
\cc\b(B\times_A B) \ar@{-{ >>}}[r]_-{\cc\b(f_{0})} & \cc\b(B)}}
\]
is a pullback in $\B$. Now the inclusion of $\B$ into $\A$ preserves and reflects all limits and moreover $\cc(f_{0})=\cc\b(f_{0})$, so that $f$ being $\cc$-central is equivalent to $f$ being $\cc\b$-central. The third statement follows from the fact that $\b$ preserves the pullback on the right above for any $\cc\b$-central extension $f$ in $\A$.
\end{proof}

\begin{lemma}\label{Lemma-Adjunction}
For any object $B$ of $\B$, the reflection from $\A$ to $\B$ restricts to an adjunction
\[
\xymatrix@1{{\Centr_{\cc\b}(B)} \ar@<1ex>[r]^-{\b} \ar@{}[r]|-{\perp} & {\Centr_{\cc}(B).} \ar@<1ex>[l]^-{\supset}}
\]
Hence the functor $\b$ preserves universal central extensions:
\[
\b\bigl(u^{\cc\b}_{B}\colon \U(B,\cc\b)\to B\bigr)\,\cong\, \bigl(u^{\cc}_{B}\colon \U(B,\cc)\to B\bigr),
\]
for any $\cc$-perfect object $B$.
\end{lemma}
\begin{proof}
First of all, by Lemma~\ref{Lemma-Nesting}.2, $\Centr_{\cc}(B)\subset (\B\downarrow B)$ is a subcategory of $\Centr_{\cc\b}(B)\subset (\A\downarrow B)$.

Suppose that $g\colon{C\to B}$ is a $\cc\b$-central extension. Applying the functor $\b$, we obtain the extension $\b(g)=g\comp\eta_{C}^{\b}\colon{\b(C)\to B}$, which is $\cc\b$-central as a quotient of~$g$. Being an extension in $\B$, $\b(g)$ is $\cc$-central by Lemma~\ref{Lemma-Nesting}.2.

Finally, as any left adjoint functor, $\b$ preserves initial objects.
\end{proof}

\begin{proposition}\label{Proposition-Comparison}
Suppose that $\C\subset \B\subset \A$ is a chain of inclusions of Birkhoff subcategories (with the left adjoints written as in~\eqref{Triangle}) of a semi-abelian category~$\A$. If~$B$ is a $\cc$-perfect object of $\B$ then we have the exact sequence
\[
\xymatrix{0 \ar[r] & [\U(B,\cc\b),\U(B,\cc\b)]_{\B} \ar@{{ |>}->}[r] & \H_{2}(B,\cc\b)\ar@{ >>}[r] & \H_{2}(B,\cc) \ar[r] & 0,}
\]
and $u^{\cc\b}_{B}=u^{\cc}_{B}$ if and only if $[\U(B,\cc\b),\U(B,\cc\b)]_{\B}$ is zero.
\end{proposition}
\begin{proof}
By Lemma~\ref{Lemma-Adjunction} and Theorem~\ref{Theorem-Universal-Central-Extension}, if $B$ is a $\cc$-perfect object of $\B$ then the comparison morphism between the induced universal central extensions gives rise to the $3\times 3$ diagram in Figure~\ref{3x3}.
\end{proof}
\begin{figure}[h]
\scalebox{.9}{
\xymatrix{[\U(B,\cc\b),\U(B,\cc\b)]_{\B} \ar@{{ |>}->}[d] \ar@{=}[r] & [\U(B,\cc\b),\U(B,\cc\b)]_{\B} \ar@{{ |>}->}[d]\ar@{ >>}[r] & 0 \ar@{{ |>}->}[d] \\
\H_{2}(B,\cc\b) \ar@{ >>}[d] \ar@{{ |>}->}[r] & \U(B,\cc\b) \ar@{ >>}[d]_-{\eta_{\U(B,\cc\b)}^{\b}} \ar@{ >>}[r]^-{u^{\cc\b}_{B}} & B \ar@{=}[d]\\
\H_{2}(B,\cc) \ar@{{ |>}->}[r] & \U(B,\cc) \ar@{ >>}[r]^-{u^{\cc}_{B}} & B}}
\caption{Proof of Proposition~\ref{Proposition-Comparison}}\label{3x3}
\end{figure}

\section{The universal central extension condition}\label{Section-UCE}
We now prove a classical recognition result for universal $\B$-central extensions. To do so, we shall need that $\A$ satisfies the universal central extension \UCE\ (see Definition~\ref{Condition-UCE} below). We show that this condition is not only sufficient but in some sense also necessary (Proposition~\ref{Proposition-UCE-domain}). We shall moreover ask that $\B$ contains~$\Ab(\A)$, so that we may suitably reduce the given situation to the case of abelianisation. The examples in~\ref{Subsection-Counterexamples} explain why these conditions are not automatically satisfied. The main result we work towards is Theorem~\ref{Theorem-Homology}, which says that a $\B$-central extension $u\colon U\to A$ is universal if and only if $\H_{1}(U,\b)=\H_{2}(U,\b)=0$.

\begin{definition}\label{Condition-UCE}
Let $\A$ be a semi-abelian category with enough projectives. We say that $\A$ satisfies \defn{\UCEspecial} when the following holds: if $B$ is an $\Ab(\A)$-perfect object, and $f\colon{B\to A}$ and $g\colon{C\to B}$ are $\Ab(\A)$-central extensions, then the extension $f\comp g$ is $\Ab(\A)$-central.
\end{definition}

\begin{lemma}\label{Lemma-Composition-Central-Extensions-Relative}
Let $\A$ be a semi-abelian category with enough projectives satisfying \UCE\ and~$\B$ a Birkhoff subcategory of $\A$ that contains $\Ab(\A)$. If $u\colon{U\to A}$ is a $\B$-central extension and $v\colon{V\to U}$ is a universal $\B$-central extension then the extension $u\comp v$ is $\B$-central.
\end{lemma}
\begin{proof}
By Proposition~\ref{Proposition-Abelian-Result} both $u$ and $v$ are $\Ab(\A)$-central. Moreover, since $\Ab(\A)$ is contained in the Birkhoff subcategory~$\B$ of~$\A$, the objects $U$, $V$ and $A$ are $\Ab(\A)$-perfect. Now by \UCE, the composite $u\comp v\colon{V\to A}$ is also $\Ab(\A)$-central. Again using that $\B$ is bigger than $\Ab(\A)$ we see that $u\comp v\colon{V\to A}$ is a $\B$-central extension (cf.\ Lemma~\ref{Lemma-Nesting}.3).
\end{proof}

Under the given assumptions, $u\comp v$ is in fact universal, as shown in Proposition~\ref{Proposition-UCE-domain}.

\begin{proposition}\label{Proposition-Universal-vs-Weak-Universal-(UCE)}
Let $\A$ be a semi-abelian category with enough projectives satisfying \UCE\ and~$\B$ a Birkhoff subcategory of $\A$ that contains $\Ab(\A)$. Then in Proposition~\ref{Proposition-Universal-vs-Weak-Universal}, condition~4 implies condition 1. Hence a~$\B$-central extension $u\colon{U\to A}$ is universal if and only if its domain $U$ is $\B$-perfect and projective with respect to all $\B$-central extensions.
\end{proposition}
\begin{proof}
Suppose that $u\colon{U\to A}$ is a universal $\B$-central extension; we have to prove that every $\B$-central extension of $U$ splits. By Theorem~\ref{Theorem-Universal-Central-Extension}, $U$ admits a universal $\B$-central extension $v\colon{V\to U}$. It suffices to prove that this $v$ is a split epimorphism. By Lemma~\ref{Lemma-Composition-Central-Extensions-Relative}, the composite $u\comp v$ is $\B$-central. The weak $\B$-universality of $u$ now yields a morphism $s\colon{U\to V}$ such that $u\comp v \comp s = u$. But also $u\comp 1_{U}=u$, so that $v\comp s=1_{U}$ by the $\B$-universality of $u$, and the universal $\B$-central extension $v$ splits. The result follows.
\end{proof}

\begin{theorem}\label{Theorem-Homology}
Let $\A$ be a semi-abelian category with enough projectives satisfying \UCE\ and $\B$ a Birkhoff subcategory of $\A$ containing~$\Ab(\A)$. A $\B$-central extension $u\colon U\to A$ is universal if and only if $\H_{1}(U,\b)$ and $\H_{2}(U,\b)$ are zero.
\end{theorem}
\begin{proof}
$\Rightarrow$ If $u\colon U\to A$ is a universal $\B$-central extension then by Proposition~\ref{Proposition-Universal-vs-Weak-Universal-(UCE)} we have $\H_{1}(U,\b)=\b(U)=0$ and $U$ is projective with respect to all $\B$-central extensions. This implies that $1_{U}\colon{U\to U}$ is a universal $\B$-central extension of~$U$. Theorem~\ref{Theorem-Universal-Central-Extension} now tells us that $\H_{2}(U,\b)=0$.

$\Leftarrow$ The object $U$ is $\B$-perfect because $\b(U) = \H_{1}(U,\b)=0$; since $\H_{2}(U,\b)$ is also zero, the universal $\B$-central extension $u_{U}^{\b}\colon{\U(U,\b)\to U}$ of $U$ induced by Theorem~\ref{Theorem-Universal-Central-Extension} is an isomorphism. Proposition~\ref{Proposition-Universal-vs-Weak-Universal-(UCE)} now implies that $U\cong \U(U,\b)$ is projective with respect to all $\B$-central extensions. Another application of Proposition~\ref{Proposition-Universal-vs-Weak-Universal-(UCE)} shows that $u$ is also a universal $\B$-central extension.
\end{proof}

\begin{proposition}\label{Proposition-UCE-domain}
Let $\A$ be a semi-abelian category with enough projectives satisfying \UCE\ and~$\B$ a Birkhoff subcategory of $\A$ that contains $\Ab(\A)$. Let $f\colon{B\to A}$ and $g\colon{C\to B}$ be $\B$-central extensions. Then the composite $f\comp g$ is a universal $\B$-central extension if and only if $g$ is.
\end{proposition}
\begin{proof}
First note that when $g$ is a universal $\B$-central extension then $f\comp g$ is $\B$-central by Lemma~\ref{Lemma-Composition-Central-Extensions-Relative}. The central extensions $f\comp g$ and $g$ have the same domain, and by Proposition~\ref{Proposition-Universal-vs-Weak-Universal-(UCE)} their universality only depends on a property of this domain.
\end{proof}

Proposition~\ref{Proposition-Universal-vs-Weak-Universal-(UCE)} has the following partial converse, which shows that in some sense \UCE\ is necessary: if we want that conditions 1--5 in Proposition~\ref{Proposition-Universal-vs-Weak-Universal} are equivalent, independently of the chosen Birkhoff subcategory $\B$ with ${\Ab(\A)\subset\B\subset\A}$, then the category $\A$ \emph{must} satisfy \UCE. See also Remark~\ref{Remark-Counterexample}.

\begin{proposition}\label{Proposition-UCE-necessary}
Let $\A$ be a semi-abelian category with enough projectives and~$\B$ a Birkhoff subcategory of $\A$ containing $\Ab(\A)$. If in Proposition~\ref{Proposition-Universal-vs-Weak-Universal} conditions 1--5 are equivalent, then the following holds: if $B$ is a $\B$-perfect object, and~$f\colon{B\to A}$ and $g\colon{C\to B}$ are $\B$-central extensions, then the extension $f\comp g$ is $\B$-central. If, in particular, this happens for $\B=\Ab(\A)$, then $\A$ satisfies \UCE.
\end{proposition}
\begin{proof}
Let $B$ be a $\B$-perfect object and consider $\B$-central extensions $f\colon{B\to A}$ and $g\colon{C\to B}$. Let $u\colon {B\to A}$ be a universal $\B$-central extension. Then by Proposition~\ref{Proposition-Quotient-of-UCE} the uniquely induced comparison morphism $\overline{f}\colon{u\to f}$ is a regular epimorphism. Pulling back $g$ along $\overline{f}$ as in
\[
\xymatrix{C\times_{B}U \splitpullback \ar@{-{ >>}}[d]_-{\underline{f}} \ar@<-.5ex>@{-{ >>}}[r]_-{\overline{g}} & U \ar@<-.5ex>@{.>}[l] \ar@{-{ >>}}[r]^-{u} \ar@{-{ >>}}[d]^-{\overline{f}} & A \ar@{=}[d] \\
C \ar@{-{ >>}}[r]_-{g} & B \ar@{-{ >>}}[r]_-{f} & A}
\]
we obtain a splitting for $\overline{g}$ through Proposition~\ref{Proposition-Universal-vs-Weak-Universal}. Now the composite $u\comp \overline{g}$ is a~$\B$-central extension because its pullback $u^{*}(u\comp \overline{g})$ along $u$ is a $\B$-trivial extension, as a composite of two $\B$-trivial extensions (Subsection~\ref{Subsection-Central-Extensions}). Since $\underline{f}$ is a regular epimorphism by regularity of $\A$, the composite $f\comp g$ is a quotient of the $\B$-central extension~$u\comp \overline{g}$, hence is also $\B$-central.
\end{proof}

\UCEcapital\ allows us to obtain the following refinement of Proposition~\ref{Proposition-Comparison}.

\begin{proposition}\label{Proposition-Comparison-Plus}
Suppose that $\A$ satisfies \UCE\ and consider a chain of inclusions of Birkhoff subcategories $\Ab(\A)\subset \C\subset \B \subset \A$ with left adjoints as in~\eqref{Triangle}. If $B$ is a $\cc$-perfect object of $\B$ then $u^{\cc\b}_{\U(B,\cc)}=\eta_{\U(B,\cc\b)}^{\b}$ and
\[
[\U(B,\cc\b),\U(B,\cc\b)]_{\B}\cong\H_{2}(\U(B,\cc),\cc\b).
\]
Hence $u^{\cc\b}_{B}=u^{\cc}_{B}$ iff $\eta_{\U(B,\cc\b)}^{\b}$ is an isomorphism iff $\H_{2}(\U(B,\cc),\cc\b)$ is zero.
\end{proposition}
\begin{proof}
In view of Remark~\ref{Remark-Subobject}, $\eta_{\U(B,\cc\b)}^{\b}$ is $\cc\b$-central as a subobject of $u^{\cc\b}_{B}$; hence Proposition~\ref{Proposition-UCE-domain} implies that it is a universal $\cc\b$-central extension of~$\U(B,\cc)$. The result now follows from Theorem~\ref{Theorem-Universal-Central-Extension}.
\end{proof}

\section{Examples}\label{Section-Examples}
As mentioned in the introduction, our theory is based on the cases of groups (with respect to abelian groups) and Lie algebras over a field~$\K$ (with respect to $\K$-vector spaces). It captures results in~\cite{Gn, LP} for $\K$-Leibniz algebras (with respect to $\K$-vector spaces) and gives new results when considering the reflection from $\Leibniz_{\K}$ to $\Lie_{\K}$ (cf.\ Section~\ref{Section-Nested-Birkhoff-Subcategories}). The chain of reflections $\PXMod\to \XMod\to \Ab\XMod$ from precrossed modules to crossed modules to abelian crossed modules also corresponds to the situation considered in Section~\ref{Section-Nested-Birkhoff-Subcategories}. Thus we regain results in~\cite{ACL, ALG, ALG2, Carrasco-Homology}. 

Via Proposition~\ref{Proposition-UCE-necessary}, the very existence of working theories already shows that in those examples \UCE\ holds; and it is easy to find further examples. Thus for the rest of the paper we focus on giving counterexamples.

\subsection{Two counterexamples}\label{Subsection-Counterexamples}
Our first counterexample is borrowed from~\cite{CIP}. It shows that a category---here the category $\NAAlg_\K$ of non-associative algebras over a field $\K$, which is a variety of $\Omega$-groups---can be semi-abelian without having to satisfy \UCE. This means that $\NAAlg_\K$ does not quite match the picture sketched in Section~\ref{Section-UCE}.

We must also emphasise that \UCE\ by itself is not yet strong enough to yield results such as Proposition~\ref{Proposition-Universal-vs-Weak-Universal-(UCE)} or Theorem~\ref{Theorem-Homology}, unless $\Ab(\A)$ is contained in $\B$. Example~\ref{Counter2}, which explains this, was offered to us by George Peschke. It describes a universal~$\B$-central extension $u\colon{U\to A}$ such that $\H_{2}(U,\b)$ is not trivial---and indeed one of the assumptions of Theorem~\ref{Theorem-Homology} is violated: the Birkhoff subcategory~$\B$ of the \mbox{(semi-)}abelian category $\A$ which we shall consider is strictly smaller than $\Ab(\A)$.

\begin{example}\label{Counter1}
A \defn{non-associative algebra} $A$ is a vector space over a field $\K$ equipped with a bilinear operation $[\cdot , \cdot]\colon A \times A \to A$. Unlike for Leibniz or Lie algebras (or for associative ones), the bracket need not satisfy any additional conditions. We write $\NAAlg_\K$ for the category of non-associative algebras over~$\K$ and remark that it coincides with the category of $\hom$-Leibniz algebras of which the twisting map is trivial ($\alpha=0$ in the notations of~\cite{MS, CIP}) and with the category of magmas in the monoidal category $(\Vect_{\K}, \tensor, \K)$. Note that $\Leibniz_{\K}$, and hence also~$\Lie_\K$ and $\Vect_\K$, are subvarieties of $\NAAlg_\K$. Furthermore, an algebra is abelian if and only if it has a trivial bracket, so precisely when it ``belongs to'' $\Vect_\K$. It is easily seen that an extension $f\colon{B\to A}$ in $\NAAlg_\K$ is $\Vect_{\K}$-central if and only if its kernel is contained in the centre of $B$, the object ${Z(B)=\{z\in B\mid\text{$[z,b]=0=[b,z]$ for all $b\in B$}\}}$.
 
In~\cite{CIP} the following situation is considered: morphisms $g\colon C\to B$ and $f\colon B\to A$ where as vector spaces, $A$, $B$ and $C$ are $2$-, $3$- and $4$-dimensional with respective bases $\{a_{1},a_{2}\}$, $\{b_{1},b_{2},b_{3}\}$ and $\{c_{1},c_{2},c_{3},c_{4}\}$. Their brackets are generated by
\begin{gather*}
[a_{2},a_{1}]=a_{2},\quad [a_{2},a_{2}]=a_{1}\\
[b_{2},b_{2}]=b_{1},\quad [b_{3},b_{2}]=b_{3}, \quad [b_{3},b_{3}]=b_{2}\\
[c_{3},c_{2}]=c_{1}, \quad [c_{3},c_{3}]=c_{2}, \quad [c_{4},c_{3}]=c_{4} \quad [c_{4},c_{4}]=c_{3}
\end{gather*}
and zero elsewhere. The algebras $B$ and $A$ are $\Vect_{\K}$-perfect because $[B,B]=B$. The morphism of non-associative algebras $f$ sends $(b_{1},b_{2},b_{3})$ to $(0,a_{1},a_{2})$ and $g$ sends $(c_{1},c_{2},c_{3},c_{4})$ to $(0,b_{1},b_{2},b_{3})$. The kernel of $f$ is generated by $b_{1}$ and thus equal to the centre $Z(B)$ of $B$. Hence $f$ is $\Vect_{\K}$-central. Likewise, $\Ker(g)$ is generated by $c_{1}$ and thus equal to $Z(C)$, so that $g$ is $\Vect_{\K}$-central. On the other hand, the kernel of~$f\comp g$ contains $c_{2}$, so that $f\comp g$ cannot be $\Vect_{\K}$-central.
\end{example}

\begin{remark}\label{Remark-Counterexample}
Combining Example~\ref{Counter1} with Proposition~\ref{Proposition-UCE-necessary} we obtain a contradiction with the statement of Proposition~6.3 in~\cite{Cheng-Su}. It entails the equivalence of all conditions in Proposition~\ref{Proposition-Universal-vs-Weak-Universal}, which would mean that the category of $\hom$-Leibniz algebras over $\K$ satisfies \UCE. We know, however, that already its subvariety $\NAAlg_{\K}$ does not, which through Lemma~\ref{Lemma-Nesting} leads to a clash. It appears that the proof given in~\cite{Cheng-Su} does not explain the second half of the ``necessary condition''.
\end{remark}

\begin{example}\label{Counter2}
Let $C$ be the infinite cyclic group (with its generator written $c\in C$) and $R=\Z[C]$ the integral group-ring over $C$. We take $\A$ to be the (abelian) category ${}_{R}\Mod$ of modules over~$R$, so that $\Ab(\A)=\A$ and \UCE\ holds. We consider its full subcategory $\B$ of all $R$-modules with a trivial $C$-action; it is clearly Birkhoff in $\A$, and its reflector is determined by tensoring with the trivial $R$-module $\Z$, so that ${\b(M)=\Z\tensor_{R}M}$ for any $R$-module $M$.

Now consider a prime number $p\neq 2$ and let $M$ be the $R$-module $\bigvee_{k\geq 1}M_{k}$, where~$M_{k}$ for $k\geq 1$ is the abelian group $\Z_{p^{k}}=\Z/_{p^{k}\Z}$ equipped with the $C$-action
\[
c\cdot m=(1-p)\cdot m.
\]
Note that a natural inclusion of $R$-modules $M_{k}\to M_{k+1}$ is given by
\[
(l+p^{k}\Z)\mapsto (p\cdot l+p^{k+1}\Z).
\]
Then it may be checked that $\H_{2}(M,\b)=\H_{2}(C,M)\cong \Z_{p}\neq 0$, while $M$ is $\B$-perfect, and
\[
u\colon M\to M\colon m\mapsto p\cdot m
\]
is a universal $\B$-central extension.
\end{example}

\section*{Acknowledgements}
We would like to thank the referee, Tomas Everaert, Julia Goedecke and George Peschke for some invaluable suggestions.


\begin{thebibliography}{10}

\bibitem{ACL}
D.~Arias, J.~M. Casas, and M.~Ladra, \emph{On universal central extensions of
 precrossed and crossed modules}, J.~Pure Appl. Algebra \textbf{210} (2007),
 177--191.

\bibitem{ALG}
D.~Arias, M.~Ladra, and A.~R.-Grandje{\'a}n, \emph{Homology of precrossed
 modules}, Illinois J.~Math. \textbf{46} (2002), no.~3, 739--754.

\bibitem{ALG2}
D.~Arias, M.~Ladra, and A.~R.-Grandje{\'a}n, \emph{Universal central extensions of precrossed modules and
 {M}ilnor's relative {$K_2$}}, J.~Pure Appl. Algebra \textbf{184} (2003),
 139--154.

\bibitem{Borceux-Bourn}
F.~Borceux and D.~Bourn, \emph{Mal'cev, protomodular, homological and
 semi-abelian categories}, Math. Appl., vol. 566, Kluwer Acad. Publ., 2004.

\bibitem{Bourn1991}
D.~Bourn, \emph{Normalization equivalence, kernel equivalence and affine
 categories}, Category {T}heory, {P}roceedings {C}omo 1990 (A.~Carboni, M.~C.
 Pedicchio, and G.~Rosolini, eds.), Lecture Notes in Math., vol. 1488,
 Springer, 1991, pp.~43--62.

\bibitem{Bourn2001}
D.~Bourn, \emph{{$3\times 3$} {L}emma and protomodularity}, J.~Algebra
 \textbf{236} (2001), 778--795.

\bibitem{Bourn-Gran}
D.~Bourn and M.~Gran, \emph{Central extensions in semi-abelian categories},
 J.~Pure Appl. Algebra \textbf{175} (2002), 31--44.

\bibitem{Carrasco-Homology}
P.~Carrasco, A.~M. Cegarra, and A.~R.-Grandje{\'a}n, \emph{({C}o){H}omology of
 crossed modules}, J.~Pure Appl. Algebra \textbf{168} (2002), no.~2-3,
 147--176.

\bibitem{CIP}
J.~M. Casas, M.~A. Insua, and N.~Pachego~Rego, \emph{On universal central
 extensions of {Hom}-{L}eibniz algebras}, preprint {\texttt{arXiv:1209.6266}},
 2012.

\bibitem{Cheng-Su}
Y.~S. Cheng and Y.~C. Su, \emph{({C}o)homology and universal central extensions
 of {H}om-{L}eibniz algebras}, Acta Math. Sin. (Engl. Ser.) \textbf{27}
 (2011), no.~5, 813--830.

\bibitem{EverHopf}
T.~Everaert, \emph{Higher central extensions and {H}opf formulae}, J.~Algebra
 \textbf{324} (2010), 1771--1789.

\bibitem{EverVdL1}
T.~Everaert and T.~Van~der Linden, \emph{{B}aer invariants in semi-abelian
 categories~{I}: {G}eneral theory}, Theory Appl. Categ. \textbf{12} (2004),
 no.~1, 1--33.

\bibitem{EverVdL2}
T.~Everaert and T.~Van~der Linden, \emph{{B}aer invariants in semi-abelian categories~{II}: {H}omology},
 Theory Appl. Categ. \textbf{12} (2004), no.~4, 195--224.

\bibitem{Gn}
A.~V. Gnedbaye, \emph{Third homology groups of universal central extensions of
 a {L}ie algebra}, Afrika Math. (S{\'e}rie 3) \textbf{10} (1999), 46--63.

\bibitem{Gran-VdL}
M.~Gran and T.~Van~der Linden, \emph{On the second cohomology group in
 semi-abelian categories}, J.~Pure Appl.\ Algebra \textbf{212} (2008),
 636--651.

\bibitem{GrayVdL1}
J.~R.~A. Gray and T.~Van~der Linden, \emph{Peri-abelian categories and the
 universal central extension condition}, in preparation, 2013.

\bibitem{Higgins}
P.~J. Higgins, \emph{Groups with multiple operators}, Proc. Lond. Math. Soc.
 (3) \textbf{6} (1956), no.~3, 366--416.

\bibitem{Janelidze:Hopf}
G.~Janelidze, \emph{Galois groups, abstract commutators and {H}opf formula},
 Appl. Categ. Structures \textbf{16} (2008), 653--668.

\bibitem{Janelidze-Kelly}
G.~Janelidze and G.~M. Kelly, \emph{Galois theory and a general notion of
 central extension}, J.~Pure Appl. Algebra \textbf{97} (1994), no.~2,
 135--161.

\bibitem{Janelidze-Marki-Tholen}
G.~Janelidze, L.~M{\'a}rki, and W.~Tholen, \emph{Semi-abelian categories},
 J.~Pure Appl. Algebra \textbf{168} (2002), no.~2--3, 367--386.

\bibitem{LP}
J.-L. Loday and T.~Pirashvili, \emph{Universal enveloping algebras of {L}eibniz
 algebras and (co)ho\-mo\-lo\-gy}, Math. Ann. \textbf{296} (1993), no.~1,
 139--158.

\bibitem{MS}
A.~Makhlouf and S.~Silvestrov, \emph{Hom-algebra structures}, J.\ Gen.\ Lie
 Theory Appl. \textbf{2} (2008), no.~2, 51--64.

\bibitem{Milnor}
J.~Milnor, \emph{Introduction to algebraic {K}-theory}, Princeton University
 Press, 1972.

\bibitem{Weibel}
Ch.~A. Weibel, \emph{An introduction to homological algebra}, Cambridge Stud.
 Adv. Math., vol.~38, Cambridge Univ. Press, 1997.

\end{thebibliography}
%

\providecommand{\noopsort}[1]{}
\providecommand{\bysame}{\leavevmode\hbox to3em{\hrulefill}\thinspace}
\providecommand{\MR}{\relax\ifhmode\unskip\space\fi MR }
\providecommand{\MRhref}[2]{%
 \href{http://www.ams.org/mathscinet-getitem?mr=#1}{#2}
}
\providecommand{\href}[2]{#2}

\end{document}